\documentclass[reqno]{amsart}
\usepackage{geometry}                		
\usepackage{graphicx}				
\usepackage{amssymb, color}
\usepackage{amsmath}
\usepackage[all]{xy}
\usepackage{array} 

\theoremstyle{plain}
\usepackage{latexsym}
\usepackage{amsthm}
\usepackage{amsfonts}
\usepackage{mathrsfs}
\usepackage{lscape}
\usepackage{enumerate}
\usepackage{pst-node}
\usepackage{tikz-cd}
\usepackage{stmaryrd}
\usepackage{epsfig}
\usepackage[sc]{mathpazo}
\usepackage{verbatim}
\usepackage[latin1]{inputenc}              
\usepackage[T1]{fontenc}
\usepackage{tikz}
\usepackage{ulem}
\usetikzlibrary{positioning}
\newtheorem{defn}{Definition}
\newtheorem{thm}{Theorem}
\newtheorem{prop}[thm]{Proposition}
\newtheorem{lem}[thm]{Lemma}
\newtheorem{cor}[thm]{Corollary}

\newtheorem{eg}{Example}

\newcommand{\cblue}[1]{{\textcolor{blue}{#1}}} 

\newcommand{\cred}[1]{{\textcolor{black}{#1}}}
\newcommand{\ccred}[1]{{\textcolor{black}{#1}}}
\newcommand{\cgreen}[1]{{\textcolor{black}{#1}}}
\newcommand{\ccgreen}[1]{{\textcolor{black}{#1}}}

\title{D-critical loci for local toric Calabi-Yau 3-folds}
\author{Sheldon Katz}
\address{Department of Mathematics, University of Illinois at Urbana-Champaign, Urbana IL 61801 USA}
\email{katzs@illinois.edu}
\author{Yun Shi}
\address{CMSA, Harvard, Cambridge MA 02138 USA}
\email{yshi@cmsa.fas.harvard.edu}
\begin{document}
\maketitle{}
\begin{abstract}
	The notion of a d-critical locus is an ingredient in the definition of motivic Donaldson-Thomas invariants by \cite{BJM19}. There is a canonical d-critical locus structure on the Hilbert scheme of dimension zero subschemes on local toric Calabi-Yau 3-folds. This is obtained by truncating the $-1$-shifted symplectic structure on the derived moduli stack \cite{BBBBJ15}. In this paper we show the canonical d-critical locus structure has critical charts consistent with the description of Hilbert scheme as a degeneracy locus \cite{BBS13}. In particular, the canonical d-critical locus structure is isomorphic to the one constructed in \cite{KS21} for local $\mathbb{P}^2$ and local $\mathbb{F}_n$.  
\end{abstract}
\section{Introduction}
Donaldson-Thomas (DT) theory was introduced in \cite{Tho00} as an enumerative theory which gives a virtual count of stable coherent sheaves with fixed topological invariants on certain 3-folds, including Calabi-Yau threefolds. Motivic DT theory was introduced in \cite{KS}, and it produces an invariant in the monodromic Grothendieck ring which categorifies the classical DT invariant. Later, a general formalism for motivic DT invariants was developed in \cite{BJM19}.

A d-critical locus structure is a main ingrediant in the definition of motivic DT invariant in \cite{BJM19}. In general there is a canonical d-critical locus structure on the moduli space of coherent sheaves/perfect complexes on a Calabi-Yau 3-fold. This follows from the fact that the derived moduli spaces of coherent sheaves/perfect complexes on a Calabi-Yau 3-fold has a $-1$-shifted symplectic structure.  See \cite{PTVV13} for the case of compact Calabi-Yau 3-folds, and \cite{BD19} for noncompact Calabi-Yau 3-folds. Then \cite{BBJ}, \cite{BBBBJ15} show that there is a d-critical locus structure obtained by truncating the $-1$-shifted symplectic structure from derived geometry. 

Now let $X$ be a local toric Calabi-Yau 3-fold $\omega_S$, the total space of the canonical bundle of a smooth, complete toric surface $S$. In \cite{KS21}, we constructed an explicit d-critical locus structure on $\operatorname{Hilb}^n(\omega_S)$ for $S=\mathbb{P}^2$ or $S=\mathbb{F}_n$.  Our construction was based on the presentation of $\operatorname{Hilb}^n(\mathbb{C}^3)$ as a degeneracy locus \cite{BBS13}, and explicitly checking the compatibility of local sections on intersections of critical charts.  We also asked whether our d-critical locus structure agrees with the canonical one obtained from derived geometry.

In this paper, we show that the canonical d-critical locus structure on $\operatorname{Hilb}^n(X)$ has critical charts isomorphic to the charts $(\operatorname{Hilb}^n(\mathbb{C}^3), \operatorname{NHilb}^n(\mathbb{C}^3), W, i)$ described in Section~\ref{sec:dcrit}. In particular, the canonical d-critical locus structure on $\operatorname{Hilb}^n(\omega_S)$ is equivalent to the one constructed in \cite{KS21} for $S=\mathbb{P}^2$ or $S=\mathbb{F}_n$. 

While we were finishing up this work, we become aware of the recent preprint \cite{RS}, in which our main result of Section~\ref{sec:explicit} was proven by direct construction. Our approach uses the Whitehead theorem in derived geometry (Theorem \ref{whitehead}). We construct an explicit model for the derived stack together with a map to the derived moduli stack constructed in \cite{TV07}, and check that this map induces an isomorphism of cotangent complexes.  We expect that our strategy will be applicable more generally to produce other explicit constructions of derived moduli stacks.

\subsection{Outline of the paper}
In Section~\ref{sec:background} we review material used in this paper on derived geometry and d-critical locus structures. In Section~\ref{sec:explicit} we give an explicit description of the derived structure on the moduli stack of zero dimensional sheaves on $\mathbb{C}^3$. In Section~\ref{sec:toric} we deduce the superpotential from the $-1$-shifted sympletic structure of the derived moduli stack, and use it to prove the main result for local toric Calabi-Yau 3-folds.

\subsection{Notation} All schemes in this paper are assumed to be separated and of finite type over $\mathbb{C}$. We use $X$ to denote a smooth quasi-projective Calabi-Yau 3-fold. For a smooth surface $S$, we denote the total space of its canonical bundle by $\omega_S$, and the projection from $\omega_S$ to $S$ by $\pi: \omega_S\rightarrow S$. We use $\mathcal{C}$ to denote a dg-category. Our typical example is the dg-category of complexes of coherent sheaves on $X$. We denote by $\mathcal{C}^c$ the subcategory of \cred{pseudo-perfect} objects of $\mathcal{C}$. Given an algebra $B$, we denote the homotopy category of the dg category of $B$-modules by $\mathcal{C}_{dg}(B)$. 

\subsection{Acknowledgements}
We would like to thank Nachiketa Adhikari \ccred{and} Aron Heleodoro for helpful conversations, and especially Tony Pantev for helpful conversations about derived algebraic geometry and for suggesting improvements to the paper.  The research of the first-named author is supported by NSF grant DMS-1802242. The second author would like to thank Dhyan Aranha \ccred{and} Ningchuan Zhang for helpful conversations on derived geometry, and Ben Davison for teaching her the bimodule resolution at MSRI. The work is done while the second author is a postdoc at CMSA, Harvard. She would like to thank CMSA for the excellent working environment.

\section{background} \label{sec:background}
\subsection{Background on derived geometry}
In this section, we review the materials on derived geometry which we will use in this paper. We start with the notion of derived moduli stacks. We use the notion of derived stacks defined in \cite{TV06} and \cite{GR10a}.


Let $k$ be a fixed base ring. Fix the base model category $\mathcal{R}$ to be the symmetric monoidal model category of simplicial commutative k-modules. Denote the affine objects in this category by $k-D^-Aff$. One endows $k-D^-Aff$ with its \'{e}tale model topology. Then for the model site $(k-D^-Aff, et)$, \cite{TV06} defines the model category $k-D^-Aff^{\sim, et}$ of stacks on the model site.

\begin{defn} (Definition 2.2.2.14 \cite{TV06})
	A $D^-stack$ is an object $F\in k-D^-Aff^{\sim, et}$ which is a stack in the sense of Definition 1.3.2.1 \cite{TV06}.
\end{defn}

Let $\mathcal{C}$ be a dg-category. There is a notion of moduli of objects in $\mathcal{C}$ defined in \cite{TV07}.

\begin{defn} (\cite{TV07})
	Define a simplicial presheaf: $\mathcal{M}_\mathcal{C}: cdga_k\rightarrow sSet$ by 
	\begin{equation*}
		\mathcal{M}_\mathcal{C}(A)=Map_{dg-Cat}(\mathcal{C}^{op}, \widehat{A}_{pe}),
	\end{equation*}
	where $Map_{dg-Cat}$ is the mapping space of model categories, and $\widehat{A}_{pe}$ is the subcategory of $Int(A-Mod)$ consisting of perfect objects. 
\end{defn}
By Lemma 3.1 in \cite{TV07}, $\mathcal{M}_\mathcal{C}$ is a $D^-$ stack.
 
In this paper we work with dg-categories of complexes of coherent sheaves with compact support on local Calabi-Yau 3-folds. In particular we work with moduli of pseudo-perfect objects in a Calabi-Yau 3-category. We recall the relevant definition given in \cite{BD19}.
\begin{defn}(\cite{BD19})
	\label{Def: noncommCY}
	A non-commutative Calabi-Yau of dimension $d$ is a (very) smooth dg category $\mathcal{C}$ equipped with a Calabi-Yau structure of dimension $d$. 
\end{defn}
Here, a Calabi-Yau structure of dimension $d$ is a negative cyclic chain $\theta: k[d]\rightarrow HC^-(\mathcal{C})$ satisfying a certain non-degeneracy condition, see \cite{BD19}. Given such a non-commutative Calabi-Yau $\mathcal{C}$, the moduli of pseudo-perfect objects in $\mathcal{C}$ is defined in the following definition. Following the notions in \cite{GR10a}, let $PreStk$ be the $(\infty, 1)$ category of all prestacks defined in \cite{GR10b}. The objects are all admissible functors from the category of derived affine schemes to the $\infty$-category of spaces. 
\begin{defn} [Example 3.7, \cite{BD19}]
	The moduli space of objects $\mathcal{M}_\mathcal{C}$ in a compactly generated dg-category $\mathcal{C}$ is the prestack given on every affine $U$ by 
	\begin{equation*}
		\mathcal{M}_{\mathcal{C}}(U)=Map_{dg-Cat}(\mathcal{C}^c, Perf(U)),
	\end{equation*}
	where $\mathcal{C}^c$ is the subcategory of pseudo-perfect objects of $\mathcal{C}$, and $Perf(U)$ is the category of perfect complexes on $U$. 
\end{defn}


By Proposition 3.4 in \cite{TV07} and Example 3.7 in \cite{BD19}, the moduli space $\mathcal{M}_{\mathcal{C}}$ has the following universal property:
\begin{prop}(Proposition 3.4, \cite{TV07})
	\label{Univfam}
	\begin{equation*}
		Map(F, \mathcal{M}_{\mathcal{C}})=Map_{dg-Cat_{op}}(L_{pe}(F), \mathcal{C})
	\end{equation*}
	for $F\in k-D^-Aff^{\sim, et}$,
	and 
	\begin{equation*}
		\widehat{A}_{pe}\simeq L_{pe}(\underline{Spec}{A}).
	\end{equation*}
\end{prop}
Here $\widehat{A}_{pe}$ is the dg category of $A$ modules, and $L_{pe}(F)$ is defined by $L_{pe}(F):=(Holim_i\widehat{A_i}_{pe})^{op}\in Ho(dg-Cat)$ for any $F\in D^-St(k)$.
The mapping space in the above proposition can be realized as a bimodule over the dg-categories $L_{pe}(F)$ and $\mathcal{C}$ by the following theorem:
\begin{thm}[\cite{Toe06b}]
	\label{Moritaequ}
	Let $\mathcal{D}_1$ and $\mathcal{D}_2$ be two dg-categories, and let $M(\mathcal{D}_1, \mathcal{D}_2)$ be the category of right quasi-representable $\mathcal{D}_1\otimes \mathcal{D}_2^{op}$ modules and quasi-isomorphisms between them. Then there exists a natural weak equivalence of simplicial sets:
	\begin{equation*}
		Map(\mathcal{D}_1, \mathcal{D}_2)=N(M(\mathcal{D}_1, \mathcal{D}_2)), 
	\end{equation*}
	where $N(M(\mathcal{D}_1, \mathcal{D}_2))$ is the nerve of the category $M(\mathcal{D}_1, \mathcal{D}_2)$.
\end{thm}

Finally we recall the Whitehead theorem for derived stacks, the main tool we will use from derived geometry. We follow the formalism of Gaitsgory-Rozenblyum. Let $PreStk_{def}\subset PreStk$ be the full subcategory spanned by objects that admit a deformation theory. \ccred{We denote} the categories of classical affine schemes and classical prestacks by $^{cl}Sch^{aff}$ and $^{cl}PreStk$.

\begin{thm}[\cite{GR10b}, Proposition 8.3.2]
	\label{whitehead}
	Let $f:X_1\rightarrow X_2$ be a map between objects of $PreStk_{def}$. Let $X_{0, cl}$ be an object in $^{cl}PreStk$. Let $g_i:X_{0, cl}\rightarrow X_i$ be pseudo-nilpotent embeddings such that $f\circ g_1=g_2$. Suppose also that for any $S\in ^{cl}Sch^{aff}$, and any map $x_0:S\rightarrow X_{0, cl}$, the induced map 
	\begin{equation*}
		T^*_{x_2}(X_2)\rightarrow T^*_{x_1}(X_1)
	\end{equation*}
is an isomorphism, where $x_i=g_i\circ x_0$.  Then $f$ is an isomorphism. 
\end{thm}

\subsection{Background on d-critical locus structures} \label{sec:dcrit}
D-critical locus structures are a main ingredient in the definition of motivic DT invariants \cite{BJM19}. In this section, we recall the definition of a d-critical locus structure in \cite{Joy15}.

Let $Y$ be a $\mathbb{C}$-scheme locally of finite type. 

\begin{thm} (\cite{Joy15}, Theorem 2.1)
	\label{thm1}
	There exists a sheaf $\mathcal{S}_Y$ of $\mathbb{C}$ vector spaces, uniquely characterized by two properties. 
	
	(i) Suppose $R\subset Y$ is a Zariski open subset of $Y$, and $i:R\hookrightarrow U$ a closed embedding in some smooth scheme $U$. Define the sheaf of ideals $I_{R, U}$ by the following exact sequence of vector spaces on $R$.
	\begin{equation*}
		0\rightarrow I_{R, U}\rightarrow i^{-1}(O_U)\rightarrow O_Y|_R\rightarrow 0.
	\end{equation*}
	
	Then there is an exact sequence of sheaves of vector spaces on $R$:
	\begin{equation*}
		0\rightarrow \mathcal{S}_Y|_R\xrightarrow{\iota_{R, U}} \frac{i^{-1}(O_U)}{I_{R, U}^2}\xrightarrow{d}\frac{i^{-1}(T^*U)}{I_{R, U}\cdot i^{-1}(T^*U)}\hspace{1mm},
	\end{equation*}
	where $\iota_{R, U}$ is a morphism of sheaves of vector spaces, and $d$ is induced by the exterior derivative.
	
	(ii) Let $R\subset S\subset Y$ be Zariski open inclusions, and $i: R\hookrightarrow U$, $j: S\hookrightarrow V$ closed embeddings in smooth schemes $U$ and $V$. Let $\Phi: U\to V$ be a morphism satisfying $\Phi\circ i=j|_R$. Then the following diagram commutes:
	\[\begin{tikzcd}
		0\arrow{r} & \mathcal{S}_Y|_R \arrow{r}{\iota_{S, V}|_R} \arrow[swap]{d}{id} & \frac{j^{-1}(O_V)}{I^2_{S, V}}|_R\arrow{r}{d} \arrow{d}{i^{-1}(\Phi^\sharp)} & \frac{j^{-1}(T^*V)}{I_{S, V}\cdot j^{-1}(T^*V)}|_R \arrow{d}{i^{-1}(d\Phi)}\\
		0\arrow{r} & \mathcal{S}_Y|_R \arrow{r}{\iota_{R, U}} & \frac{i^{-1}(O_U)}{I^2_{R, U}}\arrow{r}{d} & \frac{i^{-1}(T^*U)}{I_{R, U}\cdot i^{-1}(T^*U)}\hspace{1mm}.
	\end{tikzcd}
	\]
\end{thm}
Let $\mathcal{S}_Y^0$ be the kernel of the composition 
\begin{equation*}
	\mathcal{S}_Y\rightarrow O_Y\rightarrow O_{Y_{red}}\hspace{1mm},
\end{equation*}
where the map $\mathcal{S}_Y\rightarrow O_Y$ is locally defined by composing $\iota_{R, U}$ with $i^{-1}(O_U)\rightarrow O_Y|_R$. 

Then the sheaf $\mathcal{S}_Y$ has a canonical decomposition 
\begin{equation*}
	\mathcal{S}_Y\simeq \mathbb{C}_Y\oplus \mathcal{S}_Y^0\hspace{1mm}.
\end{equation*}

\begin{defn} (\cite{Joy15} Definition 2.5)
	\label{dcrit}
	An algebraic d-critical locus over $\mathbb{C}$ is a pair $(Y, s)$, where $Y$ is a $\mathbb{C}$-scheme and $s\in H^0(\mathcal{S}_Y^0)$ such that the following is satisfied: for every point $y\in Y$, there is a Zariski open neighborhood $R$ of $y$ with a closed embedding $i:R\hookrightarrow U$ into a smooth scheme $U$, such that $i(R)=\{df=0\}\subset U$ for $f: U\rightarrow \mathbb{C}$ a regular function on $U$. Furthermore, $\iota_{R, U}(s|_R)=i^{-1}(f)+I_{R, U}^2$. 
\end{defn}
Using the notation in the definition, the charts $(R, U, f, i)$'s are called critical charts of $(Y, s)$. \cred{For the similar definition of a d-critical stack, see Section 3.2 in \cite{BBBBJ15}.}

\begin{eg} [Presentation of $\operatorname{Hilb}^n(\mathbb{C}^3)$ as a degeneracy locus \cite{BBS13}]
	\label{HilbnC3}
	
	To a subscheme $Q\subset \mathbb{C}^3$ with Hilbert polynomial $P_Q=n$ we can associate 
	an $n$-dimensional vector space $V_n=H^0(O_Q)$, three pairwise commuting linear maps 
	\begin{equation*}
		X, Y, Z: V_n\to V_n
	\end{equation*}
	defined as multiplication by $x$, $y$, $z\in \mathbb{C}[x, y, z]$, and a vector $v\in V_n$ corresponding to $1\in H^0(O_Q)$.  The vector $v$ is  cyclic for the action of $\mathbb{C}[X, Y, Z]$ on $V_n$: $\mathbb{C}[X,Y,Z]\cdot v=V_n$.  
	
	Now consider the space of triples of $n\times n$ matrices and a vector in $V_n$: 
	\begin{equation*}
		Hom(V_n, V_n)^3\times V_n. 
	\end{equation*}
	This space is a quasiprojective variety and admits a $GL_n$ action induced from the action of $GL_n$ on $V_n$. The character $\chi: GL_n\rightarrow \mathbb{C}^*$ defined by $\chi(g)=\operatorname{det}(g)$ defines a linearization of the trivial bundle. Let $U$ be the stable locus of the linearization. It turns out that $U$ consists of the points $(X, Y, Z, v)$ where $v$ is cyclic for the action of $X, Y, Z$. Consider the GIT quotient 
	\begin{equation*}
		\operatorname{NHilb}^n(\mathbb{C}^3):=Hom(V_n, V_n)^3\times V_n\sslash GL_n=U/GL_n
	\end{equation*}
	with respect to this linearization. 
	Let $W:\operatorname{NHilb}^n(\mathbb{C}^3)\rightarrow \mathbb{C}$ be the function on $\operatorname{NHilb}^n(\mathbb{C}^3)$ defined by $W(X, Y, Z, v)=\operatorname{tr}([X, Y]Z)$. The condition $\{dW=0\}$ is equivalent to the condition that $X$, $Y$ and $Z$ pairwise commute. Then 
	the locus $\{dW=0\}$ is isomorphic to $\operatorname{Hilb}^n(\mathbb{C}^3)$. Then 
\[
(\operatorname{Hilb}^n(\mathbb{C}^3), \operatorname{NHilb}^n(\mathbb{C}^3), W, i)
\] 
itself is a critical chart for $\operatorname{Hilb}^n(\mathbb{C}^3)$ which defines a d-critical locus structure on $\operatorname{Hilb}^n(\mathbb{C}^3)$.
	\end{eg}
Note that a construction similar to Example \ref{HilbnC3} also gives a presentation of the moduli stack of length $n$ sheaves on $\mathbb{C}^3$ as a degeneracy locus of a regular function on a smooth stack.

 Using this description, in \cite{KS21} we constructed a d-critical locus structure on $Hilb^n(\omega_S)$ for $S=\mathbb{P}^2$ or $S=\mathbb{F}_n$. In general, one obtains a d-critical locus structure if there is a $-1$-shifted symplectic structure on the derived moduli stack.
\begin{thm}(\cite{BBBBJ15}, Theorem 3.18)
	\label{Thm: BBBBJ15}
	Let $(\mathbf{M}, \omega_{\mathbf{M}})$ be a $-1$-shifted symplectic derived Artin $\mathbb{C}$-stack, and $M=t_0(\mathbf{M})$ the corresponding classical Artin $\mathbb{C}$-stack. Then there exists a unique d-critical structure $s\in H^0(S_M^0)$ on $M$, making $(M, s)$ into a d-critical stack with the property that:
	
	(a) For each point $p\in M$, there exists a smooth $\mathbb{C}$-scheme $U$ with dimension $dimH^0(\mathbb{L}_X|_p)$, a point $t\in U$, a regular function $f:U\rightarrow \mathbb{A}^1$ with $d_{dR}f|_t=0$, so that $T:=Crit(f)\subset U$ is a closed $\mathbb{C}$-subscheme with $t\in T$ and a morphism $\psi: T\rightarrow X$ which is smooth of relative dimension $dim H^1(\mathbb{L}_X|_p)$ with $\psi(t)=p$.
	
	(b) Let $s_T$ be the unique section in  $H^0(S_T^0)$ with $\iota_{T, U}(s_T)=i^{-1}(f)+I_{T, U}^2$, and $(T, s_T)$ is an algebraic d-critical locus. Then $s(T, \psi)=s_T$ in $H^0(S_T^0)$.
\end{thm}

It is shown in \cite{PTVV13} that there is a $-1$ shifted symplectic structure for the derived moduli stack of perfect complexes on a compact Calabi-Yau 3-fold. For local toric Calabi-Yau 3-folds, we need to use the analogous result for a Calabi-Yau 3-category following \cite{BD19}. Recall the definition of non-commutative Calabi-Yau given in Definition \ref{Def: noncommCY}.

\begin{thm}(\cite{BD19}, Theorem 5.5)
	\label{Thm: BD19}
	Given a non-commutative Calabi-Yau $(\mathcal{C}, \theta)$ of dimension $d$, the moduli space of pseudo-perfect objects $M_\mathcal{C}$ has an induced symplectic form of degree $2-d$. 
\end{thm}

We consider the case when $d=3$, and denote this 2-form of degree $-1$ by $\omega$. 

\section{An explicit description of moduli stack of zero dimensional sheaves on $\mathbb{C}^3$} \label{sec:explicit}

In this section, we give an explicit description of the derived moduli stack of zero dimensional sheaves on $\mathbb{C}^3$. While we were finishing up this project, we become aware of the recent preprint \cite{RS}, in which the main theorem of this section was proven by a direct construction.

Our approach is through the use of the Whitehead theorem for derived stacks reviewed in the previous section. We first exhibit a derived stack with the right classical truncation and a  self-dual cotangent complex of the expected form. Then we construct a universal family on this derived stack, and use it to construct a map to $\mathcal{M}_{\mathcal{C}}$. Finally we show that this map induces an isomorphism on cotangent complexes.  

\subsection{The derived stack $\mathfrak{X}$ and its cotangent complex}
Let $A_n=k[X^0(i, j), Y^0(i, j), Z^0(i, j)]$ for $1\leq i\leq n$, $1\leq j\leq n$ \footnote{We are mostly concerned with the case $k=\mathbb{C}$, but everything goes through for an arbitrary algebraically closed field of characteristic zero.}.
Let $W\in A_n$ be the potential defined by $W=tr(X^0[Y^0, Z^0])$.  It is natural to expect the Darboux form of \cite{BBJ} associated to $W$ to be relevant.
 Explicitly, $\partial W$ induces a linear map from $A_n^{3n^2}$ to $A_n$, and we have the associated Koszul complex:
\begin{equation*}
0\rightarrow \bigwedge^{3n^2}(A_n^{3n^2})\rightarrow...\rightarrow \bigwedge^3(A_n^{3n^2})\xrightarrow{d^{-3}}\bigwedge^2(A_n^{3n^2})\xrightarrow{d^{-2}}\bigwedge^1(A_n^{3n^2})\xrightarrow{d^{-1}=\partial W} A_n\rightarrow 0
\end{equation*}
Denote this complex by $A_n^\bullet$. Following the notation in \cite{BBJ}, we denote the generators in degree $i$ by $x^i_1,...,x^{i}_{m_i}$. Then $A_n^\bullet$ is a cdga free over $A_n^\bullet(0)=A_n$ generated by $3n^2$ generators in degree $-1$. We write the generators of $A_n^{3n^2}$ in the form of $X^{-1}(i, j)$, $Y^{-1}(i, j)$, $Z^{-1}(i, j)$, where $d(X^{-1}(i, j))=(Y^0Z^0-Z^0Y^0)^T(i, j)$ and similarly for $Y^{-1}(i, j)$, $Z^{-1}(i, j)$.


The module of K\"ahler differentials $\Omega^1_{A_n(0)}$ is generated by $d_{dR}X^0(i, j)$, $d_{dR}Y^0(i, j)$, $d_{dR}Z^0(i, j)$. Then by Example 2.3 in \cite{BBJ} we have $\Omega_{A_n^\bullet}^1[1]$ is generated by 
$d_{dR}X^0(i, j)$, $d_{dR}Y^0(i, j)$, $d_{dR}Z^0(i, j)$, $d_{dR}X^{-1}(i, j)$, $d_{dR}Y^{-1}(i, j)$, $d_{dR}Z^{-1}(i, j)$ as an $A_n^\bullet$-module. The differential in $\Omega^1_{A_n^\bullet}$ is given by
	\begin{equation}
		\label{d-1}
		d^{-1}(d_{dR}X^{-1})=(Y^0d_{dR}Z^0)^T+(d_{dR}Y^0Z^0)^T-(Z^0d_{dR}Y^0)^T-(d_{dR}Z^0Y^0)^T.
	\end{equation}


Now consider the cotangent complex of the derived stack $\mathfrak{X}_n=[{\bf Spec}(A_n^\bullet)/GL(n)]$. This is \cred{described in terms of
\begin{equation*}
h^*\mathbb{L}^\bullet_{\mathfrak{X}_n}=\Omega^1_{A_n^\bullet}\xrightarrow{d^0} \mathfrak{g}^\vee\otimes O_{Spec(A_n^\bullet)},
\end{equation*}
where $h:{\bf Spec}(A_n^\bullet)\to \mathfrak{X}_n$ is the canonical map.}
Here the map $d^0$ is induced by the $GL(n)$ action on $Spec(A_n^\bullet)$\ccred{.} See Section 3.3 for more details. 
Note that $\Omega^1_{A_n^\bullet}$ has generators in degree $-1$ and $0$, hence at this point $\mathbb{L}^\bullet_{\mathfrak{X}_n}$ is generated by elements in degree $-1, 0$ and $1$, so $\mathbb{L}^\bullet_{\mathfrak{X}_n}$ cannot be self-dual with a shift of $-1$. It is natural to use the Darboux form for Artin stacks from \cite{BBBBJ15}, making the cotangent complex self-dual by adding an additional $n^2$ generators $z^{-2}_1,..., z^{-2}_{n^2}$ in degree $-2$ to $A_n^\bullet$. The map $d^{-2}$ is again induced by the action of $GL(n)$, and it is the dual of $d_0$. After this fix, we have 
\begin{equation}
\label{cotcompd}
\mathbb{L}^\bullet_{\mathfrak{X}_n}\simeq O_{\mathfrak{X}_n}^{n^2}\rightarrow O_{\mathfrak{X}_n}^{3n^2}\rightarrow O_{\mathfrak{X}_n}^{3n^2}\rightarrow O_{\mathfrak{X}_n}^{n^2}.
\end{equation}
concentrated in degree $-2$, $-1$, $0$ and $1$ as a complex of $O_{\mathfrak{X}_n}$-modules.

\subsection{Universal family on $\mathfrak{X}_n$ and a map to $\mathcal{M}_{\mathcal{C}^c_{dg}(k[x, y, z])}$}

To apply the Whitehead Theorem, we need to provide a map $\mathfrak{X}_n\to \mathcal{M}_{\mathcal{C}^c_{dg}(k[x, y, z])}$ and then check that it induces an isomorphism of cotangent complexes.   An $A_n^\bullet\otimes k[x,y,z]$ module of length $n$ will induce such a map.   One could hope that endowing a projective $A_n^\bullet$-\cred{module} of length $n$ with a  $k[x,y,z]$-\cred{module} structure would suffice, but that turns out to be too optimistic.  Let's proceed naively at first and then modify the argument.

We could consider the 
$A_n^\bullet$ module $F_n:= A^\bullet\otimes_k V_n$, where $V_n$ is an $n$-dimensional vector space and try to define a $k[x,y,z]$-module structure on $F_n$ by letting $x,y,z$ be represented by $X^0,Y^0,Z^0$, respectively.  But that only works over $t_0(A^\bullet)$, since $X^0,Y^0,Z^0$ do not pairwise commute.  But we only need to check this commutativity up to homotopy.  The generators of $A^\bullet$ in degree $-1$ provide the required homotopy, but then higher homotopies are required because the commutators are not independent but satisfy the Jacobi identity.  These are handled by the generators of $A^\bullet$ in degree $-2$.  There are no higher homotopies.

More formally, we resolve $k[x, y, z]$ by a (noncommutative) dga $D^\bullet$ which is free as an algebra.  The free generators we need are generators $x,y,z$ in degree~0 (to map to the commuting $x,y,z$), degree $-1$ generators $u,v,w$ to map to the commutators, and a generator $t$ in degree $-2$ to implement the Jacobi identity.  The differentials of $D^\bullet$ are given by
\begin{equation}
\begin{split}
dx=dy=dz=0,\ du=yz-zy,\ dv=zx-xz,\ dw = xy-yx,\\ dt=(xu-ux)+(yv-vy)+(zw-wz).
\end{split}
\end{equation}
We replace $k[x,y,z]$ by $D^\bullet$ and give $F_n$ the structure of a $D^\bullet$-module. 

As before, we let $x,y$ and $z$ act on $F_n$ as $X^0,Y^0$, and $Z^0$, respectively.   To describe the actions of $u,v$, and $w$, we first let $\{e_i\}$ denote a basis for $V_n$ and then put 
\begin{equation}
u\cdot (1\otimes e_i)=\sum_j X^{-1}(j, i)\otimes e_j,
\end{equation} 
and extend the action of $u$ to all of $F_n$ by $A^\bullet$-linearity. Here \ccred{$X^{-1}(i,j)$ are the generators} of $A^\bullet$ in degree $-1$ \ccred{which were introduced in the previous section}. We have an analogous definition for the actions of $v$ and $w$.  By construction, the actions of $u,v$, and $w$ commute with $d$.

As for the action of $t$, we first identify the $n^2$ generators of $A^{-2}$ with symbols $T(i,j)$, so that 
\begin{equation}
dT(i,j)=\left(
[X^0,X^{-1}]+[Y^0,Y^{-1}]+[Z^0,Z^{-1}]
\right)(i,j).
\end{equation}  We then put
\begin{equation}
t\cdot (1\otimes e_i)=\sum_jT(i, j)\otimes e_i.
\end{equation}
and extend the action of $t$ to all of $F_n$ by $A^\bullet$-linearity.
By construction, the action of $t$ commutes with $d$.

This completes the description of the $D^\bullet$-module structure on $F_n$. 

\smallskip
To simplify notation, we denote $\mathcal{C}_{dg}^c(k[x, y, z])$ by $\mathcal{C}$, its subcategory of length $n$ sheaves by $\mathcal{C}^n$, and the substack of $\mathcal{M}_\mathcal{C}$ parametrizing length $n$ sheaves by $\mathcal{M}^n_{\mathcal{C}}$. Next we use the universal family $F_n$ to construct a $1$-morphism in $Map(\mathfrak{X}_n, \mathcal{M}^n_{\mathcal{C}})$. By Proposition \ref{Univfam}, we only need to construct a $1$-morphism in $Map_{dg-Cat_{op}}(L_{pe}(\mathfrak{X}_n), \mathcal{C}^n)$.

Let $\mathcal{D}_1$ and $\mathcal{D}_2$ in Theorem \ref{Moritaequ} be $L_{pe}(\mathfrak{X}_n)$ and $\mathcal{C}^n$ respectively. Since the universal family $F_n$ naturally \cred{carries} a $GL_n$ action, it defines a right quasi-representable $L_{pe}(\mathfrak{X}_n)\otimes (\mathcal{C}^n)^{op}$-module, hence defines a $1$-morphism in the mapping space $Map(\mathfrak{X}_n, \mathcal{M}^n_{\mathcal{C}})$.

We will compare the cotangent complexes in the next section.

\subsection{Comparison of the cotangent complexes of $\mathfrak{X}_n$ and $\mathcal{M}^n_{\mathcal{C}}$} 
We first review a bimodule resolution for the Jacobi algebra of a quiver which is Calabi-Yau. This resolution will be needed for computing the cotangent complex of $\mathcal{M}^n_{\mathcal{C}}$. We omit $n$ in the notation in this section. 

Recall that $[Spec(A))/GL(n)]$ can be realized as the moduli stack of $n$-dimensional representations of the following quiver $Q$:

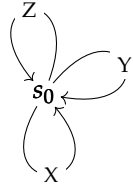
\begin{center}
	\begin{tikzpicture}

		\node[circle, minimum size=0.1pt,inner sep=0pt,outer sep=0pt] (b) {$s_0$} edge [in=-50,out=-120,loop, scale=3] node[midway, fill=white, scale=0.8]{X} ();
		\node[circle, minimum size=0.1pt,inner sep=0pt,outer sep=0pt] (b) {$s_0$} edge [in=-10,out=50,loop, scale=3] node[midway, fill=white, scale=0.8]{Y} ();
		\node[circle,  minimum size=0.1pt,inner sep=0pt,outer sep=0pt] (b) {$s_0$} edge [in=130,out=70,loop, scale=3] node[midway, fill=white, scale=0.8]{Z} ();

	\end{tikzpicture}
\end{center}
Now the path algebra $kQ_1$ is isomorphic to $k<x, y, z>$, and we
let $w=tr(x[y, z])$ be the potential in $kQ_1$. We denote $kQ_1/(\partial w)$ by $C$.  
Since $C$ is a Calabi-Yau algebra, the following complex from \cite{Gin} gives a bimodule resolution of $C$: 
\begin{equation}
\label{bimodres}
0\rightarrow (C\otimes C)^R\xrightarrow{\alpha^{-2}} C\otimes E^*\otimes C\xrightarrow{\alpha^{-1}}C\otimes E\otimes C\xrightarrow{\alpha^0} C\otimes C\rightarrow C\rightarrow 0,
\end{equation}
where $E$ is the vector space of edges in the quiver, and all the tensors below are over the idempotent subring $R\subset B$.

We have 
\begin{equation}
\label{alpha0}
\alpha^0: 1\otimes x\otimes 1\mapsto x\otimes 1-1\otimes x,
\end{equation}
\begin{equation*}
\alpha^{-1:} 1\otimes x^*\otimes 1\mapsto y\otimes z\otimes 1+1\otimes y\otimes z- z\otimes y\otimes 1-1\otimes z\otimes y,
\end{equation*}
with similar formulae for $\alpha^{-1}(1\otimes y^*\otimes 1)$ and $\alpha^{-1}(1\otimes z^*\otimes 1)$, and

 \begin{equation*}
	\alpha^{-2}:1\otimes 1\mapsto 1\otimes x^*\otimes x-x\otimes x^*\otimes 1+1\otimes y^*\otimes y-y\otimes y^*\otimes 1+1\otimes z^*\otimes z-z\otimes z^*\otimes 1.
\end{equation*}
\bigskip

Recall the following notion from \cite{BD19}, \cite{GR10a}. Let $\mathcal{A}$ be a monoidal dg-category. Let $\mathcal{B}$ be an $\mathcal{A}$-module category. 
For any object $b\in \mathcal{B}$, denote the right adjoint to the the functor $\_\otimes b: \mathcal{A}\rightarrow\mathcal{B}$ by $\underline{Hom}_{\mathcal{A}}(b, \_)$.  Let $U$ be an affine derived scheme, and define the natural transformation
\begin{equation*}
\Upsilon: QCoh(-)^*\rightarrow IndCoh(-)^!
\end{equation*}
by 
\begin{equation*}
	\Upsilon_U:-\otimes \omega_U: QCoh(U)\rightarrow IndCoh(U).
\end{equation*}
Then one has:
\begin{prop} [\cite{BD19}, Proposition 3.3]
			\label{Equ:CotanCplx}
	\begin{equation*}
		\mathbb{T}(\mathcal{M}_\mathcal{C})[-1]\simeq \underline{Hom}_{\mathcal{M}_{\mathcal{C}}}(\Upsilon F, \Upsilon F).
	\end{equation*}
\end{prop}

\cred{Denote} $GL_n$ by $G$. Take $\mathcal{A}$ to be 
\begin{equation*}
	IndCoh(\mathfrak{X})=IndCoh([{\bf Spec}(A^\bullet)/G]),
\end{equation*} and take $\mathcal{B}$ to be 
\begin{equation*}
	IndCoh([{\bf Spec}(A^\bullet\otimes_kk[x, y, z])/G]).
\end{equation*}
Note that $IndCoh([{\bf Spec}(A^\bullet)/GL_n])\simeq D^G(A^\bullet)$, where $D^G(A^\bullet)$ is the equivariant derived infinity category of $A^\bullet$-modules.
Take the object $b$ to be $F$. We denote the map induced by the universal object $F$ by $f: \mathfrak{X}\rightarrow \mathcal{M}_{\mathcal{C}}$. We are trying to show $g: \mathbb{T}_{\mathfrak{X}}^\bullet\rightarrow f^*\mathbb{T}(\mathcal{M}_\mathcal{C})$ is an equivalence. Note that $g$ is an equivalence in $D^G(A^\bullet)$ if its projection as a morphism in \ccred{its} homotopy category $[D^G(A^\bullet)]$ is an isomorphism. So we only need to work with $[D^G(A^\bullet)]$. Since we already have the map $g$, we only need to show the isomorphism after pullback to $[D(A^\bullet)]$, which is \cred{$D^b(A^\bullet)$},
the usual derived category of $A^\bullet$-modules. We use the same notation for the projection of every object in the homotopy category and its lift in $[D(A^\bullet)]$.

Then by Proposition \ref{Equ:CotanCplx}, we have  
\begin{equation*}
	f^*\mathbb{T}(\mathcal{M}_\mathcal{C})[-1]\simeq \mathbb{R}Hom_{A^\bullet\otimes_kk[x, y, z]}(F, F)
\end{equation*}
as an isomorphism in $D^b(A^\bullet\otimes_kk[x, y, z])$. We use Ginzburg's bimodule resolution (\ref{bimodres}) to \cred{give} an explicit representative of $RHom_{A^\bullet\otimes_kk[x, y, z]}(F, F)$.

Tensoring the resolution (\ref{bimodres}) by $B^\bullet:= A^\bullet\otimes k[x, y, z]$ over $k[x, y, z]$ on the left and on the right, we get a bimodule resolution of $A^\bullet \otimes k[x, y, z]$:

\begin{equation}
	\label{bimodresB}
	0\rightarrow (B^\bullet\otimes B^\bullet)^R\xrightarrow{\alpha^{-2}} B^\bullet\otimes E^*\otimes B^\bullet\xrightarrow{\alpha^{-1}}B^\bullet\otimes E\otimes B^\bullet\xrightarrow{\alpha^0} B^\bullet\otimes B^\bullet\rightarrow B^\bullet\rightarrow 0,
\end{equation}
Then we have the following quasi-isomorphism in $D^b(A^\bullet\otimes k[x, y, z])$  (\cred{See an analogous discussion in} Section 2.1 in \cite{Shi18}). 

\begin{equation}
	\label{Ext}
	\begin{split}
		RHom_{A^\bullet\otimes k[x, y, z]}(F, F)&\simeq F^*\otimes F\xrightarrow{\alpha^{-2}} F^*\otimes E^*\otimes F\xrightarrow{\alpha^{-1}}F\otimes E\otimes F^*\xrightarrow{\alpha^{-1}}F\otimes F^*\\
		&\simeq 
		\begin{array}{c}
			F^*\otimes F\\
		\end{array}
		\xrightarrow{\alpha^{-2}}
		\begin{array}{c}  F^*\otimes x^* \otimes F\\ 
			\oplus F^*\otimes y^* \otimes F \\ 
			\oplus F^*\otimes z^* \otimes F \\
		\end{array}
		\xrightarrow{\alpha^{-1}}
		\begin{array}{c} F\otimes x \otimes F^*\\ 
			\oplus F\otimes y \otimes F^* \\ 
			\oplus F\otimes z \otimes F^* \\
		\end{array}
		\xrightarrow{\alpha^{0}}
		\begin{array}{c}
			F\otimes F^*\\
		\end{array}
	\end{split}
\end{equation}

We denote this complex by $L^\bullet$. The generators of $RHom_{A^\bullet\otimes k[x, y, z]}(F, F)$ as an $A^\bullet$ module are in degrees $0$, $1$, $2$, and $3$. 

\begin{lem}
	\label{isocotanv2}
		\begin{equation*}
			[g]: \mathbb{T}^\bullet_{\mathfrak{X}}\rightarrow f^*\mathbb{T}(\mathcal{M}_\mathcal{C})\simeq RHom_{A^\bullet\otimes k[x, y, z]}(F, F)[1]
		\end{equation*}
		is an isomorphism in $D^b(A^\bullet)$.
\end{lem}

\begin{proof}
	We can view $F$ as an $A^\bullet\otimes D^\bullet$-module with a $\operatorname{GL}(n)$-action.  
	We represent $\mathbb{T}^\bullet_{\mathfrak{X}}$ by the dual of (\ref{cotcompd}).   We describe the pullback of $\mathbb{T}^\bullet_{\mathfrak{X}}$ to $\operatorname{Spec}(A^\bullet)$ as the cone of the map $\mathcal{O}_{{\bf Spec}(A^\bullet)}^{n^2}\rightarrow \mathbb{T}^\bullet_{{\bf Spec}(A^\bullet)}$ associated to the infinitesimal $\operatorname{GL}(n)$-action.  
	
Since $\mathbb{T}^\bullet_{\mathfrak{X}}$ is a complex of free modules, we can regard $[g]$ as \cgreen{an} actual map between complexes. We write down the definition of $[g]$ on generators explicitly. Let $\tilde{A}^\bullet$ and $\tilde{D}^\bullet$ be the graded rings underlying $A^\bullet$ and $D^\bullet$ respectively.  We consider generators \cgreen{of $\mathbb{T}^\bullet_{\mathfrak{X}}$ in degree 0, e.g.\ }the vector field $d/dX^0(i,j)$ on $\operatorname{Spec}(\tilde{A}^\bullet)$. Viewing $F$ as a $\tilde{A}^\bullet\otimes\tilde{D}^\bullet$-module $\tilde{F}$, this vector field defines a first-order deformation $\mathcal{F}$ of $\tilde{F}$, a module over $\tilde{A}^\bullet\otimes\tilde{D}^\bullet\otimes k[\epsilon]/(\epsilon^2)$, flat over $k[\epsilon]/(\epsilon^2)$.  The action of $\tilde{D}^\bullet$ on $\mathcal{F}$ is defined by letting $t,u,v,w,y,z$ act on $\mathcal{F}$ exactly as they did for $F$, while the action of $x$ is replaced by matrix multiplication by $X^0+\epsilon E_{ij}$, where $E_{ij}$ is the matrix whose $(i,j)$ entry is 1 with all other entries vanishing.
	
	We get an extension
	\begin{equation}\label{x0ijext}
		0\to \tilde{F} \to \mathcal{F}\to \tilde{F}\to 0.
	\end{equation}
	We resolve $\tilde{F}$ by free $G^\bullet=\tilde{A}^\bullet\otimes\tilde{D}^\bullet$-modules
	\begin{equation}
		\label{resG}
		F^\bullet: 0\rightarrow G^\bullet\otimes F \rightarrow G^\bullet \otimes E^* \otimes F \rightarrow F \otimes E \otimes G^\bullet 
		\rightarrow F\otimes G^\bullet \rightarrow \tilde{F}\rightarrow 0.
	\end{equation}
	Let $\{f_i\}_i$ be a set of basis \cgreen{elements} for $F$.
	Using this resolution, the extension class of (\ref{x0ijext}) is represented by the map $F \otimes E \otimes G^\bullet \rightarrow F\otimes G^\bullet$ induced by $f_i^*\otimes x^*\otimes f_j$.   This term is naturally identified with the corresponding term in the Ginzburg resolution.
	
	We can perform a similar calculation for the other vector fields corresponding to the generators of $A^\bullet$, as well as for the infinitesimal $\operatorname{GL}(n)$ corresponding to $\mathcal{O}_{{\bf Spec}(A^\bullet)}^{n^2}$ mentioned above. \cgreen{We find that $d/dX^{-1}(i,j)$ maps to $f_i\otimes x \otimes f_j^*$, $d/dz^{-2}_{ij}$ maps to $f_i\otimes f_j^*$, and similarly for the vector fields associated with \ccgreen {$Y$ and $Z$}.} By the definition of pushforward of vector fields and the Kodaira-Spencer map to $RHom_{A^\bullet\otimes k[x, y, z]}(F, F)$, we see that the above calculation indeed is the map $[g]$ on generators.
	
		Thus we have the following diagram:
	\[\begin{tikzcd}	
		O_{\mathfrak{X}}^{n^2}\arrow{r}{d^{-1}}\arrow{d}{\varphi^{-1}} & O_{\mathfrak{X}}^{3n^2}\arrow{r}{d^{0}}\arrow{d}{\varphi^{0}}  & O_{\mathfrak{X}}^{3n^2}\arrow{r}{d^{1}}\arrow{d}{\varphi^{1}} & O_{\mathfrak{X}}^{n^2}\arrow{d}{\varphi^2}\\
		F^*\otimes F \arrow{r}{\alpha^{-1}} & F^*\otimes E^*\otimes F\arrow{r}{\alpha^{0}} & F\otimes E\otimes F^*\arrow{r}{\alpha^1}&F\otimes F^*
	\end{tikzcd},
	\]
where the top row is identified with $\mathbb{T}^\bullet_{\mathfrak{X}}$, and $[g]$ is explicitly given by the maps $\varphi^i$. Using the definition of $\alpha^i$ and $d^i$, it is easy to see $\varphi^i$'s indeed define an isomorphism between complexes.
\end{proof}
We immediately conclude the main result of this section.
\begin{thm}
	The universal family $F$ induces an isomorphism of derived stacks
\begin{equation*}
	f: \mathfrak{X}\rightarrow \mathcal{M}_\mathcal{C}.
\end{equation*}	
\end{thm}
\begin{proof}
	This follows from Lemma \ref{isocotanv2} and Theorem \ref{whitehead}. Indeed, lemma \ref{isocotanv2} implies that $[g]$ is an isomorphism in $[D(A^\bullet)]$, hence $g:\mathbb{T}_{\mathfrak{X}}^\bullet\rightarrow f^*\mathbb{T}(\mathcal{M}_\mathcal{C})$ is an equivalence. Then Theorem \ref{whitehead} implies the desired result.
\end{proof}

\section{d-critical locus structure on $Hilb^n(\omega_S)$}\label{sec:toric}

In this section we work with local toric Calabi-Yau 3-folds. 
Let $S$ be a smooth and projective toric surface, and let $\omega_S$ be the total space of the canonical bundle of $S$. Consider the dg-category of complexes of coherent sheaves with compact support on $\omega_S$. 
We denote the stack parametrizing length $n$ sheaves on $\omega_S$ by $\mathcal{M}_{\omega_S}^n$, and its classical trunction $t_0(\mathcal{M}_{\omega_S}^n)$ by $M_{\omega_S}^n$.  By Theorem \ref{Thm: BBBBJ15}, there is a canonical d-critical structure on $M_{\omega_S}^n$ as a truncation of the $-1$-shifted symplectic structure on $\mathcal{M}_{\omega_S}^n$. We show that this d-critical locus structure has critical charts $(\operatorname{Hilb}^n(\mathbb{C}^3), \operatorname{NHilb}^n(\mathbb{C}^3), W, i)$ as in Example \ref{HilbnC3}. Again, we omit $n$ in the notation.

\smallskip\noindent
{\bf Remark.} The authors of \cite{RS} provided an analytic description of a d-critical locus structure on quot schemes of a compact Calabi-Yau threefold.  We have restricted attention to local toric surfaces so that we can work algebraically.

\smallskip
By Theorem 2.10 and Corollary 2.11 in \cite{BBBBJ15}, a critical chart for the truncated d-critical locus structure is constructed from a minimal standard form open neighbourhood 
\begin{equation*}
	(R, \gamma: {\bf U}=\mathbf{Spec}(R)\rightarrow {\bf X}, \widetilde{p})
\end{equation*}
of $p$, for $p$ a point in $M_{\omega_S}^n$. Choosing $H$, $\Phi$ and $\phi$ as in Theorem 2.10 \cite{BBBBJ15}, the critical chart is defined by $(Crit(H), Spec(R(0)), H, i)$.

The following Proposition shows that every point $p\in \mathcal{M}_{\omega_S}$ has a standard form open neighbourhood equivalent to $\mathfrak{X}$, where $\mathfrak{X}$ is defined in Section 3.1. 
\begin{prop}
	\label{cover}
	The stack $\mathcal{M}_{\omega_S}$ admits a cover by open substacks isomorphic to $\mathfrak{X}$.
\end{prop}
\begin{proof}
	Denote the projection map by $\pi:\omega_S\rightarrow S$. We recall our strategy employed in \cite{KS21}.  We need only show that there exists an open cover of $S$ by open subsets $\{V_\alpha\}$ isomorphic to $\mathbb{C}^2$ with the property that for any finite subset of $Z\subset S$, $Z$ is contained in one of our open subsets $V_\alpha$.  Applying this to the projection to $S$ of the support of a zero-dimensional sheaf $G$, we see that $G$ is contained in $M^n_{U_\alpha}$, where $U_\alpha=\pi^{-1}(V_\alpha)\simeq\mathbb{C}^3$.
	
	We describe such a cover $\{V_\alpha\}$ by an inductive procedure. Recall that all toric surfaces can be obtained by successively blowing up $\mathbb{P}^2$ or $\mathbb{F}_n$ at torus invariant points. We denote this process by:
	\begin{equation*}
		S=X_m\xrightarrow{f_n} X_{m-1} \rightarrow...\xrightarrow{f_1} X_0,
	\end{equation*}
	where $X_0$ is $\mathbb{P}^2$ or $\mathbb{F}_n$, and $f_j$ represents the blowup of $X_{j-1}$ at a torus invariant point $q_{j-1}$. We know from \cite{KS21} that $X_0$ admits a cover $\{V_0^i\}$, such that the charts $\{M^n_{\pi^{-1}(V_0^i)}\}$ cover $M^n_{\omega_{X_0}}$. Now assume that $X_{j-1}$ admits such a cover. Let $I$ be the index subset consisting of $i$ such that $q_{j-1}\in V_{j-1}^i$. Then we have $f_j^{-1}(V_{j-1}^i)\simeq \mathbb{C}^2$ for $i\notin I$, and $f_j^{-1}(V_{j-1}^{i'})\simeq Bl_{q_{j-1}}\mathbb{C}^2$ for $i'\in I$. Let $V_j^{i',k}$ be the open subset of $f_j^{-1}(V_{j-1}^{i'})$ obtained by removing the fiber over $k\in \mathbb{P}^1$. Then $\{f_j^{-1}(V_{j-1}^i)\}_{i\notin I}\cup \{V_j^{i', k}\}_{i'\in I}$ defines a cover of $X_i$ as required. 
\end{proof}
Now we write down the pullback of the $2$-form $\omega$ to ${\bf Spec}(A^\bullet)$ explicitly: 
\begin{lem}
		\begin{equation*}
			\begin{split}
				\gamma^*\omega=&tr(d_{dR}X^{-1}\wedge (d_{dR}X^{0})^T+d_{dR}Y^{-1}\wedge (d_{dR}Y^0)^T+d_{dR}Z^{-1}\wedge (d_{dR}Z^0)^T)\\
				=&\sum_i\sum_jd_{dR}X^{-1}(i, j)\wedge d_{dR}X^0(i, j)+\sum_i\sum_jd_{dR}Y^{-1}(i, j)\wedge d_{dR}Y^0(i, j)\\
				&+\sum_i\sum_jd_{dR}Z^{-1}(i, j)\wedge d_{dR}Z^0(i, j)
			\end{split}
		\end{equation*}
\end{lem}

\begin{proof}
	By Corollary 2.5 and Proposition 5.3 in \cite{BD19}, the closed $2$-form $\omega$ is identified with the Serre pairing:
	\begin{equation*}
		\label{Serrepair}
		\underline{End}_{\mathcal{M}_{\omega_S}^n}(F_\mathcal{C})^{\otimes 2}\xrightarrow{\circ}\underline{End}_{\mathcal{M}_{\omega_S}^n}(F_\mathcal{C})\xrightarrow{tr}O_{\mathcal{M}_{\omega_S}^n}.
	\end{equation*}
	
	
	Hence $\gamma^*(\omega)$ is identitied with the Serre pairing on $RHom_{A^\bullet\otimes k[x, y, z]}(F, F)$.
	Consider the presentation of $RHom_{A^\bullet\otimes k[x, y, z]}(F, F)$ by the bimodule resolution (\ref{Ext}). Recall that we \cgreen{denoted the resolution of $F$ in (\ref{resG})} by $F^\bullet$.
	We choose an element $\phi$ in $Hom(F^\bullet, F^\bullet[1])$, and write $M=(M_X, M_Y, M_Z)^T$ with $M_X(i, j)$ the matrix corresponding to \cgreen{the basis element} $f_i^*\otimes X^*\otimes f_j$, and similarly for $M_Y$ and $M_Z$. We choose another element $N$ in $Hom(F^\bullet, F^\bullet[2])$, and write $\psi=(N_X, N_Y, N_Z)$ with $N_X(i, j)$ the matrix corresponding to \cgreen{the basis element} $f_j\otimes X\otimes f_i^*$ and similarly for $N_Y$ and $N_Z$. Now we have 
	\begin{equation*}
		tr(N\circ M)=tr(N_XM_X+N_YM_Y+N_ZM_Z)
	\end{equation*} 
	We get the same formular if we take the trace of composition of an element in $Hom(F^\bullet, F^\bullet[2])$ and $Hom(F^\bullet, F^\bullet[1])$.
	Using the \cred{isomorphisms} $\varphi^{0}$ and $\varphi^{1}$ in the commutative diagram, we obtain the $2$-form in degree $-1$ as desired. 
\end{proof}

By Theorem 2.10 (a) in \cite{BBBBJ15}, the critical chart is defined by $(Crit(H), Spec(A(0)), H, i)$, where $H=-\Phi$, and 
\begin{equation}
	\label{relation}
	\begin{split}
		\omega\sim (d_{dR}(\phi), 0, 0, 0...)\\
		d_{dR}(\Phi)+d\phi=0.
	\end{split}
\end{equation}
\begin{lem}
	We can choose the superpotential induced by $\omega$ by $\Phi=-trX_0[Y_0, Z_0]$.
\end{lem}
\begin{proof}
	We will choose $\phi$ as in Theorem 2.10 in \cite{BBBBJ15} and then check the relations (\ref{relation}).
	We take $\phi\in (\Omega^1_{A^\bullet})^{-1}$ to be
		\begin{equation*} tr(X^{-1}(d_{dR}X^{0})^T+Y^{-1}(d_{dR}Y^{0})^T+Z^{-1}(d_{dR}Z^{0})^T).
		\end{equation*}
		It is easy to see that $\omega\sim(d_{dR}(\phi), 0, 0...)$.
	
	Take $\Phi=-tr(X^0[Y^0, Z^0])$. Then 
	\begin{equation*}
		\begin{split}
			d_{dR}(\Phi)=&-tr((d_{dR}X^0)Y^0Z^0-(d_{dR}X^0)Z^0Y^0)-tr(X^0(d_{dR}Y^0)Z^0-X^0Z^0(d_{dR}Y^0))\\
			&-tr(X^0Y^0(d_{dR}Z^0)-X^0(d_{dR}^0Z)Y^0)
		\end{split}
	\end{equation*}
	On the other hand, 
		\begin{equation*}
			\begin{split}
				d(\phi)&=tr((Y^0Z^0-Z^0Y^0)^T(d_{dR}X^0)^T+(Z^0X^0-X^0Z^0)^T(d_{dR}Y^0)^T+(X^0Y^0-Y^0X^0)^T(d_{dR}Z^0)^T)
				)
			\end{split}
		\end{equation*}
	As a result, we have 
	\begin{equation*}
		d_{dR}(\Phi)+d\phi=0.
	\end{equation*}
\end{proof}
Take the Hamiltonian $H=-\Phi=trX_0[Y_0, Z_0]$, we have 
\begin{thm}
	\label{Crit}
	The d-critical locus structure truncated from the $-1$-shifted symplectic structure on $\mathcal{M}^n_{\omega_S}$ have critical charts all isomorphic to \ccgreen{$(Crit(H), \operatorname{Spec}(A_n^\bullet(0)), H, i)=(Spec(A_n/(\partial W)), Spec(A_n), W, i)$.}
\end{thm}

In \cite{KS21}, we showed that there is a d-critical locus structure on $Hilb^n(\omega_S)$ for $S=\mathbb{P}^2$ or $S=\mathbb{F}_n$. 
\begin{thm}\cite{KS21}
	\label{KS21}
	Suppose that $S=\mathbb{P}^2$ or $S=\mathbb{F}_n$.  Then $\operatorname{Hilb}^n(\omega_S)$ has a d-critical locus structure with critical charts all isomorphic to $(\operatorname{Hilb}^n(\mathbb{C}^3),N\operatorname{Hilb}^n(\mathbb{C}^3),W,i)$.
\end{thm} 

This was done by explicitly checking that the local sections $W_i$'s agree up to $I^2_{\operatorname{Hilb}^n(\mathbb{C}^3), \operatorname{NHilb}^n(\mathbb{C}^3)}$ on intersections. Thus for $S=\mathbb{P}^2$ or $S=\mathbb{F}_n$, we have 
\begin{cor}
	The pullback of d-critical locus structure in Theorem \ref{Crit} to $Hilb^n(\omega_S)$ is equivalent to the d-critical locus structure in Theorem \ref{KS21}.
\end{cor}

\end{document}